\documentclass{amsart}
\usepackage{mathptmx, amscd, amssymb, enumerate, colonequals, color}
\definecolor{chianti}{rgb}{0.6,0,0}
\definecolor{meretale}{rgb}{0,0,.6}
\definecolor{leaf}{rgb}{0,.35,0}
\usepackage[colorlinks=true, pagebackref, hyperindex, citecolor=meretale, urlcolor=leaf, linkcolor=chianti]{hyperref}

\newtheorem{theorem}{Theorem}[section]
\newtheorem{lemma}[theorem]{Lemma}

\theoremstyle{definition}

\newtheorem{example}[theorem]{Example}
\newtheorem{remark}[theorem]{Remark}
\newtheorem{question}[theorem]{Question}
\numberwithin{equation}{theorem}

\newcommand{\del}{\partial}
\newcommand{\Ann}{\operatorname{Ann}}
\newcommand{\Hom}{\operatorname{Hom}}
\newcommand{\Ext}{\operatorname{Ext}}
\newcommand{\Proj}{\operatorname{Proj}}

\renewcommand{\ge}{\geqslant}
\renewcommand{\le}{\leqslant}

\renewcommand{\tilde}{\widetilde}
\renewcommand{\to}{\longrightarrow}
\renewcommand{\mod}{\,\operatorname{mod}\,}

\newcommand{\frakm}{\mathfrak{m}}

\newcommand{\FF}{\mathbb{F}}
\newcommand{\PP}{\mathbb{P}}

\newcommand{\calC}{{\mathcal{C}}}
\newcommand{\calH}{{\mathcal{H}}}

\newcommand{\calM}{{\mathcal{M}}}

\newcommand{\calO}{{\mathcal{O}}}
\newcommand{\calX}{{\mathcal{X}}}

\begin{document}
\title{Frobenius on the cohomology of thickenings}

\author[Bhatt]{Bhargav Bhatt}
\address{IAS/Princeton University and University of Michigan}
\email{bhargav.bhatt@gmail.com}

\author[Blickle]{Manuel Blickle}
\address{Johannes Gutenberg-Universit\"at Mainz}
\email{blicklem@uni-mainz.de}

\author[Lyubeznik]{Gennady Lyubeznik}
\address{University of Minnesota}
\email{gennady@math.umn.edu}

\author[Singh]{Anurag K. Singh}
\address{University of Utah}
\email{singh@math.utah.edu}

\author[Zhang]{Wenliang Zhang}
\address{University of Illinois at Chicago}
\email{wlzhang@uic.edu}

\thanks{B.B.~was supported by NSF grants DMS~1801689, DMS~1952399, and DMS~1840234, the Packard Foundation, and the Simons Foundation~622511; M.B.~by DFG grant SFB/TRR45 and CRC326 GAUS; G.L.~by NSF grant DMS~1800355, A.K.S.~by NSF grants DMS~1801285 and DMS~2101671; and W.Z.~by NSF grant DMS~1752081. The authors are also grateful to the American Institute of Mathematics (AIM) and the Institute for Advanced Study (IAS) for supporting their collaboration. We thank Johan de Jong and Mircea Musta\c t\u a for encouragement and conversations.}

\begin{abstract}
We investigate the injectivity of the Frobenius map on thickenings of smooth varieties in projective space over a field of positive characteristic. We obtain uniform bounds---i.e., independent of the characteristic---on the thickening that ensures an injective Frobenius map when the projective variety is a smooth complete intersection or an arbitrary projective embedding of an elliptic curve. Our bounds are sharp in the case of hypersurfaces, and in the case of elliptic curves.
\end{abstract}
\maketitle

%%%%%%%%%%%%%%%%%%%%%%%%%%%%%%%%%%%%%%%%%%%%%%%%%%%%%%%%%%%%%%%%%%%%%%%%
\section{Introduction}
%%%%%%%%%%%%%%%%%%%%%%%%%%%%%%%%%%%%%%%%%%%%%%%%%%%%%%%%%%%%%%%%%%%%%%%%

Let $X$ be a closed subscheme of $\PP^n$ defined by an ideal $I$ of $S\colonequals\FF[x_0,\dots,x_n]$, where~$\FF$ is a field of characteristic $p>0$. We use $X_t$ to denote the \emph{$t$-th thickening} of $X$, i.e., the subscheme defined by~$I^t$. Suppose $\FF$ has characteristic $p\ge t$, consider the composition
\[
\CD
S/I @>>> S/I^{[p]} @>>> S/I^t,
\endCD
\]
where $I^{[p]}$ is the ideal generated by $p$-th powers of generators of $I$, the first map is induced by the Frobenius endomorphism of $S$, and the second is the canonical surjection. For $k$ an integer, consider the induced map on cohomology groups
\[
\tilde{F}_t\colon H^k(X,\calO_X) \to H^k(X_t,\calO_{X_t}).
\]
This paper is motivated by the following question:

\begin{question}
\label{question:main}
Let $X$ be a smooth subvariety of $\PP^n$, over a field of characteristic $p>0$. Does there exist an integer $t\le p$, depending only on $\dim X$, such that for each $k$, the map~$\tilde{F}_t$ as above is injective?
\end{question}

We prove that the integer $t=\dim X+1$ suffices in two cases:

\begin{theorem}
\label{theorem:main}
Let $X$ be a smooth subvariety of $\PP^n$, over a field of characteristic $p>0$. Suppose that either
\begin{enumerate}[\quad\rm(1)]
\item\label{theorem:main:elliptic} $X$ is an elliptic curve, or
\item\label{theorem:main:hypersurface} $X$ is a hypersurface, and $p\ge n$.
\end{enumerate}
Then the map 
\[
\tilde{F}_t\colon H^{\dim X}(X,\calO_X) \to H^{\dim X}(X_t,\calO_{X_t})
\]
is injective when $t=\dim X+1$.
\end{theorem}

To give this some context, suppose $X$ is an elliptic curve over a field of positive characteristic. Then the Frobenius map
\[
\tilde{F}_1\colon H^1(X,\calO_X) \to H^1(X,\calO_X)
\]
is injective if and only if the elliptic curve is ordinary; in contrast, Theorem~\ref{theorem:main}.\ref{theorem:main:elliptic} says that~$\tilde{F}_2$ is injective independent of whether $X$ is ordinary or supersingular. When $X$ is an elliptic curve in $\PP^2$, this was proved earlier as~\cite[Theorem~4.1]{BS}, while Theorem~\ref{theorem:main}.\ref{theorem:main:hypersurface} extends the results of \cite{BS} from the Calabi-Yau case to that of all smooth hypersurfaces.

In the case of a hypersurface of characteristic $p$ in $\PP^n$, it is easy to see that the map~$\tilde{F}_p$ is injective, see Remark~\ref{remark:pth:injective}; what is striking in Theorem~\ref{theorem:main}.\ref{theorem:main:hypersurface} is that the~$n$-th thickening suffices independent of the characteristic. Moreover, the integer $n$ here is sharp: for each~$n\ge 2$ and each $d\ge n+1$, we construct a hypersurface $X$ in $\PP^n$, of degree $d$, for which
\[
\tilde{F}_{n-1}\colon H^{n-1}(X,\calO_X) \to H^{n-1}(X_{n-1},\calO_{X_{n-1}})
\]
is not injective, see Example~\ref{example:sharp}. One cannot expect uniform injectivity results for positive twists of the structure sheaf, see Example~\ref{example:positive:twist}, or without some version of the smoothness hypotheses, see Example~\ref{example:singular}.

Another affirmative answer to Question~\ref{question:main} is when $X$ is a complete intersection in $\PP^n$:

\begin{theorem}
\label{theorem:frobenius:ci:intro}
Let $X$ be a smooth complete intersection in $\PP^n$, over a field of characteristic~$p>0$. Then there exists an integer $t$, depending only on $n$ and on the degrees of the minimal defining equations, such that the map
\[
\tilde{F}_t\colon H^{\dim X}(X,\calO_X) \to H^{\dim X}(X_t,\calO_{X_t})
\]
is injective provided $p\ge t$.
\end{theorem}

While a bound on $t$ in the theorem above may indeed be obtained using Theorem~\ref{theorem:ci}, we have not attempted in the present paper to optimize this bound.

We briefly explain the genesis of Question~\ref{question:main}; it arose organically from certain calculations in the case of Calabi-Yau hypersurfaces. More precisely, the injectivity of the Frobenius map on thickenings is closely related to the $F$-pure thresholds of Musta\c t\u a, Takagi, and Watanabe~\cite{TW, MTW}, that are invariants of singularities in positive characteristic analogous to characteristic zero log canonical thresholds; for instance, for a supersingular elliptic curve~$X$ of characteristic $p$ in $\PP^2$, the injectivity of $\tilde{F}_2\colon H^1(X,\calO_X) \to H^1(X_2,\calO_{X_2})$ yields that the $F$-pure threshold of the curve is $1-1/p$, see~\cite[Remark~2.2]{BS}. Analogous assertions hold for all Calabi-Yau hypersurfaces $X$ in $\PP^n$ provided $p$ is sufficiently large compared to~$n$, see \cite[Theorem~4.1.4 and Lemma~4.5]{BS}. Given that the injectivity of the Frobenius map on thickenings is the key cohomological input in these calculations, it is then natural to formulate Question~\ref{question:main} for arbitrary smooth varieties $X$ in $\PP^n$.

\begin{remark}
The integer $t$ in Question~\ref{question:main} is, in the case of Calabi-Yau hypersurfaces, closely related to the order of vanishing of the Hasse invariant at $[X]$ on the family of all such hypersurfaces, see~\cite[Lemma~4.5]{BS}; this was investigated in depth by Ogus~\cite{Ogus}.
\end{remark}

\begin{remark}
Though we do not pursue it here, it would be interesting to understand the role of projective space in Question~\ref{question:main}; are there other natural smooth varieties that one may use instead? We thank Mircea Musta\c t\u a for highlighting this question. 
\end{remark}

%%%%%%%%%%%%%%%%%%%%%%%%%%%%%%%%%%%%%%%%%%%%%%%%%%%%%%%%%%%%%%%%%%%%%%%%
\section{Preliminaries}
%%%%%%%%%%%%%%%%%%%%%%%%%%%%%%%%%%%%%%%%%%%%%%%%%%%%%%%%%%%%%%%%%%%%%%%%

Let $S\colonequals \FF[x_0,\dots,x_n]$ be a polynomial ring over a field $\FF$ of characteristic~$p>0$, and let~$\frakm$ denote its homogeneous maximal ideal. For integer powers $q$ of $p$, set
\[
\frakm^{[q]}\colonequals(x_0^q,\dots,x_n^q)S.
\]
Ring elements and ideals considered in this paper are homogeneous under the standard grading on $S$. By the \emph{Jacobian ideal} of a polynomial~$f$, we mean the ideal $J$ generated by the partial derivatives
\[
f_{x_i}\colonequals \del\!f/\del x_i\quad\text{ for }\ 0\le i\le n.
\]
The ideal $J+fS$ is $\frakm$-primary when $\Proj S/fS$ is smooth.

More generally, if $f_1,\dots,f_c$ is a regular sequence of homogeneous forms in $S$, let~$J$ denote the ideal generated by the size $c$ minors of the matrix
\[
\begin{pmatrix}
\frac{\del\!f_1}{\del x_0} & \dots & \frac{\del\!f_c}{\del x_0}\\
\vdots & & \vdots\\
\frac{\del\!f_1}{\del x_n} & \dots & \frac{\del\!f_c}{\del x_n}
\end{pmatrix}.
\]
The condition that $\Proj S/(f_1,\dots,f_c)S$ is smooth implies that $J+(f_1,\dots,f_c)S$ is $\frakm$-primary.

\begin{lemma}
\label{lemma:hilb}
Let $f$ be a homogeneous polynomial in $S\colonequals\FF[x_0,\dots,x_n]$ such that $\Proj S/fS$ is smooth. Set $d\colonequals\deg f$. Then
\[
\frakm^{n(d-2)+d}\ \subseteq\ J+fS\,.
\]
\end{lemma}

\begin{proof}
The assertion is a statement regarding the Hilbert-Poincar\'e series of $S/(J+fS)$, and this is unaffected by enlarging $\FF$ to an infinite field, so as to use homogeneous prime avoidance as follows: the ideal $J+fS$ is $\frakm$-primary, and $J$ is generated in degree $d-1$, so~$f$ can be extended to a homogeneous system of parameters for $S$, where the parameters have degrees~$d,d-1,\dots,d-1$. The socle modulo these elements is spanned by an element of degree $d-1+n(d-2)$.
\end{proof}

The following is essentially \cite[Lemma~3.2]{BS}; a proof is sketched for convenience.

\begin{lemma}
\label{lemma:colon}
Let $S\colonequals\FF[x_0,\dots,x_n]$ be a polynomial ring over a field of characteristic $p>0$. Then, for each $q\colonequals p^e$ and each $N\ge0$, one has
\[
\frakm^{[q]}:_S\frakm^N\ =\ \frakm^{(n+1)q-n-N}+\frakm^{[q]},
\]
with the convention that $\frakm^i\colonequals S$ for $i\le0$.
\end{lemma}

\begin{proof}
The pigeonhole principle gives
\[
\frakm^{(n+1)q-n}\subseteq\frakm^{[q]},
\]
which explains one inclusion. For the other, if $s$ is a homogeneous element of $\frakm^{[q]}:_S\frakm^N$, then $\frakm^N$ annihilates the element
\[
\left[\frac{s}{x_0^q\cdots x_n^q}\right]
\]
of the local cohomology module $H^{n+1}_\frakm(S)$. If this element is nonzero, then it has degree at least $-n-N$.
\end{proof}

\begin{lemma}
\label{lemma:diagram}
Let $f\in S\colonequals\FF[x_0,\dots,x_n]$ be a homogeneous form, where $\FF$ is a field of characteristic $p$. Set $d\colonequals\deg f$, and let $t$ be an integer with $1\le t\le p$. Then there exists a commutative diagram
\[
\CD
H^n_\frakm(S/fS) @>\cong>> \Ann\big(f,\ H^{n+1}_\frakm(S)(-d)\big)\phantom{,}\\
@VV\tilde{F_t}V @VVf^{p-t}FV\\
H^n_\frakm(S/f^tS) @>\cong>> \Ann\big(f^t,\ H^{n+1}_\frakm(S)(-dt)\big),
\endCD
\]
where, in the vertical map on the right,
\[
F\colon H^{n+1}_\frakm(S)\to H^{n+1}_\frakm(S)
\]
is the map induced by the Frobenius endomorphism of $S$.

More generally, suppose $f_1,\dots,f_c$ is a regular sequence of homogeneous forms in $S$. Setting~$d\colonequals\sum\deg f_i$, one has a commutative diagram
\[
\CD
H^{n+1-c}_\frakm(S/(f_1,\dots,f_c)S) @>\cong>> \Ann\big((f_1,\dots,f_c),\ H^{n+1}_\frakm(S)(-d)\big)\\
@VV\tilde{F}_{[t]}V @VV(f_1\cdots f_c)^{p-t}FV\\
H^{n+1-c}_\frakm(S/(f_1^t,\dots,f_c^t)S) @>\cong>> \Ann\big((f_1^t,\dots,f_c^t),\ H^{n+1}_\frakm(S)(-dt)\big),
\endCD
\]
where $\tilde{F}_{[t]}$ is the map of local cohomology modules induced by
\[
\CD
S/(f_1,\dots,f_c) @>>> S/(f_1^p,\dots,f_c^p) @>>> S/(f_1^t,\dots,f_c^t),
\endCD
\]
with the first map induced by Frobenius, and the second being the canonical surjection.
\end{lemma}

\begin{proof}
For the first assertion, note that the commutative diagram
\[
\CD
0 @>>> S(-d) @>f>>S @>>> S/fS @>>> 0\\
@. @VVf^{p-t}FV @VVFV @VV\tilde{F_t}V\\
0 @>>> S(-dt) @>f^t>>S @>>> S/f^tS @>>> 0
\endCD
\]
induces a commutative diagram of local cohomology modules
\begin{equation}
\label{equation:cd}
\CD
0 @>>> H^n_\frakm(S/fS) @>>> H^{n+1}_\frakm(S)(-d) @>f>> H^{n+1}_\frakm(S) @>>>0\\
@. @VV\tilde{F_t}V @VVf^{p-t}FV @VVFV\\
0 @>>>H^n_\frakm(S/f^tS) @>>>H^{n+1}_\frakm(S)(-dt) @>f^t>>H^{n+1}_\frakm(S) @>>>0
\endCD
\end{equation}
where the rows are exact. The second assertion has a similar inductive proof.
\end{proof}

\begin{remark}
\label{remark:pth:injective}
It is immediate from the above that the map
\[
\tilde{F_{\!p}}\colon H^n_\frakm(S/fS) \to H^n_\frakm(S/f^pS)
\]
is injective: the Frobenius action on $H^{n+1}_\frakm(S)$ is injective.
\end{remark}

%%%%%%%%%%%%%%%%%%%%%%%%%%%%%%%%%%%%%%%%%%%%%%%%%%%%%%%%%%%%%%%%%%%%%%%%
\section{Hypersurfaces}
%%%%%%%%%%%%%%%%%%%%%%%%%%%%%%%%%%%%%%%%%%%%%%%%%%%%%%%%%%%%%%%%%%%%%%%%

We begin with the proof of Theorem~\ref{theorem:main}.\ref{theorem:main:hypersurface}, followed by examples illustrating that this result is sharp in multiple ways.

\begin{proof}[Proof of Theorem~\ref{theorem:main}.\ref{theorem:main:hypersurface}]
Let $f\in S\colonequals\FF[x_0,\dots,x_n]$ be a homogeneous form defining the hypersurface $X$. Set $d\colonequals\deg f$. The map
\[
\tilde{F}_t\colon H^{n-1}(X,\calO_X) \to H^{n-1}(X_t,\calO_{X_t})
\]
is precisely the map
\[
\tilde{F}_t\colon {[H^n_\frakm(S/fS)]}_0 \to {[H^n_\frakm(S/f^tS)]}_0,
\]
so taking $t=n$ in Lemma~\ref{lemma:diagram}, it suffices to prove that
\[
f^{p-n}F\colon {[H^{n+1}_\frakm(S)]}_{-d} \to {[H^{n+1}_\frakm(S)]}_{-dn}
\]
is injective. Computing local cohomology via the \v Cech complex on the elements $x_0,\dots,x_n$, a nonzero element of ${[H^{n+1}_\frakm(S)]}_{-d}$ may be expressed as
\[
\eta\colonequals\left[\frac{s}{(x_0\cdots x_n)^{q/p}}\right]
\]
for some integer power $q$ of the characteristic $p$, where $s$ is a homogeneous element of $S$ with degree $-d+(n+1)q/p$. Suppose $f^{p-n}F(\eta)=0$. Then
\[
f^{p-n}s^p\ \in\ \frakm^{[q]},
\]
whereas the assumption $\eta\neq 0$ implies that $s^p\notin\frakm^{[q]}$. Take $k$ to be the smallest integer with
\begin{equation}
\label{equation:minimal:k}
f^ks^p\ \in\ \frakm^{[q]}.
\end{equation}
Applying the $S^p$-linear derivations $\del/\del x_i$ to the above, one obtains
\[
kf_{x_i}f^{k-1}s^p\ \in\ \frakm^{[q]}\qquad\text{ for }0\le i\le n,
\]
where $f_{x_i}\colonequals \del\!f/\del x_i$. Since $1\le k\le p-n$, the image of $k$ in $\FF$ is nonzero, so
\[
Jf^{k-1}s^p\ \subseteq\ \frakm^{[q]},
\]
where $J\colonequals(f_{x_0},\dots,f_{x_n})S$ is the Jacobian ideal of $f$. It follows that
\[
(J+fS) f^{k-1}s^p\ \subseteq\ \frakm^{[q]}.
\]
Combining this with Lemma~\ref{lemma:hilb}, one obtains
\[
\frakm^{n(d-2)+d}f^{k-1}s^p\ \subseteq\ \frakm^{[q]},
\]
so Lemma~\ref{lemma:colon} gives
\[
f^{k-1}s^p\ \in\ \frakm^{(n+1)(q-d)+n}+\frakm^{[q]}.
\]
But $f^{k-1}s^p\notin\frakm^{[q]}$ by the minimality of $k$ in~\eqref{equation:minimal:k}, so the polynomial $f^{k-1}s^p$ must have degree at least $(n+1)(q-d)+n$, i.e.,
\[
(k-1)d-pd+(n+1)q\ \ge\ (n+1)(q-d)+n,
\]
which yields a contradiction since $k\le p-n$.
\end{proof}

The following example illustrates that the $n$-th thickening in Theorem~\ref{theorem:main}.\ref{theorem:main:hypersurface} is optimal:

\begin{example}
\label{example:sharp}
Fix $n\ge 2$ and $d\ge n+1$, and consider the hypersurface $X$ defined by
\[
f\colonequals x_0^d+\dots+x_n^d
\]
in $S\colonequals\FF[x_0,\dots,x_n]$, where $\FF$ is a field of characteristic $p\equiv -1\mod d$. We claim that
\[
\tilde{F}_{n-1}\colon H^{n-1}(X,\calO_X) \to H^{n-1}(X_{n-1},\calO_{X_{n-1}})
\]
is not injective. View $H^{n-1}(X,\calO_X)$ as ${[H^n_\frakm(S/fS)]}_0$, with the latter computed via the \v Cech complex on $x_0,\dots,x_n$. The hypothesis $d\ge n+1$ ensures that
\[
\eta\colonequals\left[\frac{x_0^n}{x_1\cdots x_n}\right]
\]
is a nonzero element ${[H^n_\frakm(S/fS)]}$. We claim that
\[
\tilde{F}_{n-1}(\eta)\ \in\ {[H^n_\frakm(S/f^{n-1}S)]}_0
\]
is zero. For this, it suffices to verify that
\[
x_0^{np}\ \in\ (x_1^p,\dots,x_n^p,\ f^{n-1})S.
\]
Let $p=kd-1$, for $k$ an integer. Then $np\ge (nk-1)d$, so it suffices to verify that
\[
x_0^{(nk-1)d}\ \in\ (x_1^{kd},\dots,x_n^{kd},\ f^{n-1})S.
\]
Setting $y_i\colonequals x_i^d$ for each $i$, one has $f=y_0+\dots+y_n$, so the required verification now is
\[
y_0^{nk-1}\ \in\ \big(y_1^k,\dots,y_n^k,\ (y_0+\dots+y_n)^{n-1}\big)S.
\]
Working modulo $(y_0+\dots+y_n)^{n-1}$, the element $y_0^{nk-1}$ is congruent to an element in
\[
(y_1,\dots,y_n)^{nk-n+1},
\]
which is contained in $(y_1^k,\dots,y_n^k)$, settling the claim.
\end{example}

\begin{remark}
Regarding negative twists of the structure sheaf, an injectivity result for the Frobenius action is provided by~\cite[Theorem~3.5]{BS}: Let $X$ be a smooth hypersurface of degree~$d$ in $\PP^n$, over a field of characteristic $p\ge\min\{d+1,\ nd-d-n\}$. Then
\[
\tilde{F}_1\colon H^{n-1}(X,\calO_X(-k)) \to H^{n-1}(X,\calO_X(-pk))
\]
is injective for each $k\ge 1$; it is worth emphasizing that no thickening is needed in this case.
\end{remark}

In view of Serre vanishing, one cannot expect similar uniform results when dealing with positive twists of the structure sheaf:

\begin{example}
\label{example:positive:twist}
Let $X$ be a smooth quartic hypersurface in $\PP^2$, and consider the map
\begin{equation}
\label{equation:positive:twist:1}
\tilde{F}_t\colon H^1(X,\calO_X(1)) \to H^1(X_t,\calO_{X_t}(p)).
\end{equation}
If this map is injective, then $H^1(X_t,\calO_{X_t}(p))$ is nonzero, so $p\le 4t-3$, i.e.,~$t\ge(p+3)/4$. It follows that there is no uniform $t$, i.e., independent of $p$, for which~\eqref{equation:positive:twist:1} is injective.

As such, the injectivity of the map~\eqref{equation:positive:twist:1} is equivalent to that of
\begin{equation}
\label{equation:positive:twist:2}
f^{p-t}F\colon {[H^3_\frakm(S)]}_{-3} \to {[H^3_\frakm(S)]}_{p-4t}
\end{equation}
by Lemma~\ref{lemma:diagram}. If this map is \emph{not} injective, then $f^{p-t}\in\frakm^{[p]}$; applying differential operators and imitating the proof of Theorem~\ref{theorem:main}.\ref{theorem:main:hypersurface}, one obtains $t\le (p+6)/4$. Thus,~\eqref{equation:positive:twist:2} is injective for thickenings $X_t$ with $t>(p+6)/4$.

For an explicit example, consider the hypersurface $X$ defined by $f=x_0^4+x_1^4+x_2^4$. We claim that the least $t$ such that the map~\eqref{equation:positive:twist:2} is injective is
\[
t=\begin{cases}
\frac{p+3}{4} & \text{ if } p\equiv 1\mod 4,\\
\frac{p+9}{4} & \text{ if } p\equiv 3\mod 4.
\end{cases}
\]

Suppose $p=4k+1$, it suffices to check that $f^{p-(p+3)/4}=f^{3k}\notin\frakm^{[p]}$, which holds since the monomial $x_0^{4k}x_1^{4k}x_2^{4k}$ occurs in $f^{3k}$ with a nonzero coefficient. If $p=4k+3$, one has
\[
f^{p-(p+5)/4}=f^{3k+1}\in (x_0^{4k+4},\ x_1^{4k+4},\ x_2^{4k+4})\subseteq\frakm^{[p]},
\]
so~\eqref{equation:positive:twist:2} in not injective with $t=(p+5)/4$. However,~\eqref{equation:positive:twist:2} is injective for $t=(p+9)/4$ by the bound recorded previously.
\end{example}

\begin{example}
\label{example:singular}
Consider $X$ in $\PP^2$ defined by $x_0^3-x_1^2x_2$. This hypersurface is not smooth and, indeed, the least $t$ with 
\[
\tilde{F}_t\colon H^1(X,\calO_X) \to H^1(X_t,\calO_{X_t})
\]
injective increases with the characteristic $p$ as follows:
\[
t=\begin{cases}
\frac{p+5}{6} & \text{ if } p\equiv 1\mod 6,\\
\frac{p+7}{6} & \text{ if } p\equiv 5\mod 6.
\end{cases}
\]
These are straightforward calculations using binomial expansions, and are omitted.
\end{example}

%%%%%%%%%%%%%%%%%%%%%%%%%%%%%%%%%%%%%%%%%%%%%%%%%%%%%%%%%%%%%%%%%%%%%%%%
\section{Complete intersections}
%%%%%%%%%%%%%%%%%%%%%%%%%%%%%%%%%%%%%%%%%%%%%%%%%%%%%%%%%%%%%%%%%%%%%%%%

The proof of Theorem~\ref{theorem:frobenius:ci:intro} relies on the following:

\begin{theorem}
\label{theorem:ci}
Let $S\colonequals\FF[x_0,\dots,x_n]$ be a polynomial ring over a field $\FF$ of positive characteristic $p$. Let $f_1,\dots,f_c$ be a regular sequence of homogeneous forms in $S$, defining a smooth projective variety $\Proj S/(f_1,\dots,f_c)S$. Set $d_i\colonequals\deg f_i$, and $d\colonequals\sum d_i$.

If $t$ is an integer with $t\le p$, and
\[
td_i\ \ge\ (n+1-c)(d-c)+1
\]
for each $1\le i\le c$, then the map
\[
\tilde{F}_{[t]}\colon {[H^{n+1-c}_\frakm(S/(f_1,\dots,f_c)S)]}_0\to {[H^{n+1-c}_\frakm(S/(f_1^t,\dots,f_c^t)S)]}_0,
\]
as defined in Lemma~\ref{lemma:diagram}, is injective. 
\end{theorem}

\begin{proof}
In view of Lemma~\ref{lemma:diagram}, it suffices to verify that for $t$ as in the theorem, the map
\[
(f_1\cdots f_c)^{p-t}F\colon {[H^{n+1}_\frakm(S)]}_{-d} \to {[H^{n+1}_\frakm(S)]}_{-dt}
\]
is injective when restricted to the annihilator of $(f_1,\dots,f_c)S$. Consider a nonzero element~$\eta$ of ${[H^{n+1}_\frakm(S)]}_{-d}$ that is annihilated by $(f_1,\dots,f_c)S$. Then
\[
\eta=\left[\frac{s}{(x_0\cdots x_n)^{q/p}}\right],
\]
where $q$ is a power of $p$, and $s\in S$ is homogeneous of degree $-d+(n+1)q/p$. The condition that $\eta$ is annihilated by each $f_i$ implies that $f_is\in\frakm^{[q/p]}$, and hence that
\begin{equation}
\label{equation:ci:1}
f_i^ps^p\in\frakm^{[q]}\quad\text{ for }\ 1\le i\le c.
\end{equation}
Suppose that
\[
(f_1\cdots f_c)^{p-t}F(\eta)=0,
\]
then
\[
(f_1\cdots f_c)^{p-t}s^p\ \in\ \frakm^{[q]}.
\]
Consider the partial order on $c$-tuples where $(k_1,\dots,k_c)\le (l_1,\dots,l_c)$ if $k_i\le l_i$ for each $i$; let $(k_1,\dots,k_c)$ be a minimal $c$-tuple with the property that $k_i\le p-t$ for each $i$, and
\begin{equation}
\label{equation:ci:2}
f_1^{k_1}\cdots f_c^{k_c}s^p\ \in\ \frakm^{[q]}.
\end{equation}
Applying the differential operators $\del/\del x_i$ to the above, we obtain
\begin{multline}
\label{equation:derivations}
k_1\frac{\del\!f_1}{\del x_i} f_1^{k_1-1}f_2^{k_2}\cdots f_c^{k_c}s^p\ + \dots +\
k_c\frac{\del\!f_c}{\del x_i} f_1^{k_1}\cdots f_{c-1}^{k_{c-1}}f_c^{k_c-1}s^p\\
=\ 
\Big(\frac{\del\!f_1}{\del x_i} k_1f_2\cdots f_c\ + \dots +\ \frac{\del\!f_c}{\del x_i} k_cf_1\cdots f_{c-1}\Big)
\ f_1^{k_1-1}\cdots f_c^{k_c-1}s^p\ \in\ \frakm^{[q]}
\end{multline}
for each $i$ with $0\le i\le n$. The ideal generated by the entries of the product matrix
\[
\begin{pmatrix}
\frac{\del\!f_1}{\del x_0} & \dots & \frac{\del\!f_c}{\del x_0}\\
\vdots & & \vdots\\
\frac{\del\!f_1}{\del x_n} & \dots & \frac{\del\!f_c}{\del x_n}\\
\end{pmatrix}
\begin{pmatrix}
k_1f_2\cdots f_c\\
\vdots\\
k_cf_1\cdots f_{c-1}
\end{pmatrix}
\]
contains the ideal
\[
J(k_1f_2\cdots f_c,\ \dots,\ k_cf_1\cdots f_{c-1})
\]
where $J$ is Jacobian ideal: selecting any $c$ rows of the matrix $(\del\!f_j/\del x_i)$, one may multiply on the left by the classical adjoint of the resulting $c\times c$ submatrix. Hence~\eqref{equation:derivations} gives
\[
J(k_1f_2\cdots f_c,\ \dots,\ k_cf_1\cdots f_{c-1})\ f_1^{k_1-1}\cdots f_c^{k_c-1}s^p\ \subseteq\ \frakm^{[q]}.
\]
Since $s^p\notin\frakm^{[q]}$ as $F(\eta)\neq0$, some $k_i$ must be nonzero in~\eqref{equation:ci:2}. After relabelling the elements $f_i$, assume without loss of generality that $k_1$ is nonzero. Then
\[
J f_1^{k_1-1}f_2^{k_2}\cdots f_c^{k_c}s^p\ \subseteq\ \frakm^{[q]},
\]
and using~\eqref{equation:ci:1} and~\eqref{equation:ci:2} we can moreover conclude that
\[
\big(J + (f_1,\ f_2^{p-k_2},\ \dots,\ f_c^{p-k_c})\big)\ f_1^{k_1-1}f_2^{k_2}\cdots f_c^{k_c}s^p\ \subseteq\ \frakm^{[q]}.
\]
The ideal $J+(f_1,f_2,\dots,f_c)S$ is $\frakm$-primary by the smoothness hypothesis, hence so is the ideal $J+(f_1,f_2^{p-k_2},\dots,f_c^{p-k_c})S$. We claim that
\[
\frakm^N\ \subseteq\ J+(f_1,f_2^{p-k_2},\dots,f_c^{p-k_c})S,
\]
where
\[
N\colonequals d_1+\Big(\sum_{i=2}^c d_i(p-k_i)\Big)+(n+1-c)(d-c)-n.
\]
The proof of the claim follows that of Lemma~\ref{lemma:hilb}: the ideal $J$ is generated in degree $d-c$, so after enlarging the field $\FF$, the regular sequence $f_1,f_2^{p-k_2},\dots,f_c^{p-k_c}$ can be extended to a homogeneous system of parameters for $S$ by choosing~$n+1-c$ elements from $J$, each of degree $d-c$. It follows that
\[
\frakm^N f_1^{k_1-1}f_2^{k_2}\cdots f_c^{k_c}s^p\ \subseteq\ \frakm^{[q]},
\]
and Lemma~\ref{lemma:colon} gives
\[
f_1^{k_1-1}f_2^{k_2}\cdots f_c^{k_c}s^p\ \in\ \frakm^{[q]}:_S\frakm^N\ =\ \frakm^{[q]}+\frakm^{(n+1)q-n-N}.
\]
The minimality assumption on $(k_1,\dots,k_c)$ in~\eqref{equation:ci:2} implies that
\[
f_1^{k_1-1}f_2^{k_2}\cdots f_c^{k_c}s^p\ \notin\ \frakm^{[q]},
\]
and hence that
\[
f_1^{k_1-1}f_2^{k_2}\cdots f_c^{k_c}s^p\ \in\ \frakm^{(n+1)q-n-N}.
\]
Examining degrees, one has
\[
\deg(f_1^{k_1-1}f_2^{k_2}\cdots f_c^{k_c}s^p)\ \ge\ (n+1)q-n-N,
\]
i.e.,
\[
(k_1-1)d_1+\sum_{i=2}^ck_id_i-pd+(n+1)q\ \ge\ (n+1)q-n-N,
\]
which simplifies to
\[
(n+1-c)(d-c)\ \ge\ (p-k_1)d_1.
\]
But $p-k_1\ge t$, so $(n+1-c)(d-c)\ge td_1$, which contradicts the assumption on $t$.
\end{proof}

Using Theorem~\ref{theorem:ci}, we obtain:

\begin{proof}[Proof of Theorem~\ref{theorem:frobenius:ci:intro}]
Let $X=\Proj S/(f_1,\dots,f_c)S$ for $f_i$ as in the previous theorem, and choose~$t_0$ such that 
\[
t_0d_i\ \ge\ (n+1-c)(d-c)+1
\]
for each $i$. Then for $t\ge c(t_0-1)+1$ and $p\ge t$, one has
\[
\CD
S/(f_1,\dots,f_c) @>>> S/(f_1^p,\dots,f_c^p) @>>> S/(f_1,\dots,f_c)^t @>>> S/(f_1^{t_0},\dots,f_c^{t_0}),
\endCD
\]
with the first map induced by Frobenius, and the others being canonical surjections. But
\[
\tilde{F}_{[t_0]} \colon {[H^{n+1-c}_\frakm(S/(f_1,\dots,f_c)S)]}_0 \to {[H^{n+1-c}_\frakm(S/(f_1^{t_0},\dots,f_c^{t_0})S)]}_0
\]
is injective by Theorem~\ref{theorem:ci}; it factors through
\[
\tilde{F}_t\colon {[H^{n+1-c}_\frakm(S/(f_1,\dots,f_c)S)]}_0 \to {[H^{n+1-c}_\frakm(S/(f_1,\dots,f_c)^tS)]}_0,
\]
which is therefore injective.
\end{proof}

%%%%%%%%%%%%%%%%%%%%%%%%%%%%%%%%%%%%%%%%%%%%%%%%%%%%%%%%%%%%%%%%%%%%%%%%
\section{Elliptic curves}
%%%%%%%%%%%%%%%%%%%%%%%%%%%%%%%%%%%%%%%%%%%%%%%%%%%%%%%%%%%%%%%%%%%%%%%%

It remains to settle Theorem~\ref{theorem:main}.\ref{theorem:main:elliptic}, i.e., to prove:

\begin{theorem}
\label{theorem:elliptic}
Let $X$ be an elliptic curve in $\PP^n$, over a field of characteristic $p>0$. Then the Frobenius map
\[
\tilde{F}_2\colon H^1(X,\calO_X) \to H^1(X_2,\calO_{X_2})
\]
is injective.
\end{theorem}

\begin{proof}
The statement is insensitive to replacing the ground field $\FF$ by its algebraic closure, so we assume $\FF$ is algebraically closed. For the proof, it will be convenient to generalize the construction of $\tilde{F}_2$ slightly by allowing arbitrary ambient spaces as follows:

Given an $\FF$-scheme $P$, and a closed immersion $i\colon X\to P$, write $2i(X) \subset P$ for the square-zero thickening defined by $I_X^2$; the Frobenius on $P$ induces a map $2i(X) \to X$. We call the closed immersion $(X \subset P)$ \emph{good} if the pullback
\[
H^1(X,\calO_X) \to H^1(2X,\calO_{2i(X)}),
\]
induced by the Frobenius on $P$, is injective. Thus, the identity map $X \to X$ is good exactly when the elliptic curve $X$ is ordinary. The theorem amounts to showing that the given closed immersion~$X\subset \PP^n$ is good.

First, observe that goodness descends: given closed immersions $(X \subset P)$ and $(X \subset P')$ with a map $P' \to P$ compatible with the inclusion of $X$, if $(X \subset P')$ is good, so is $(X \subset P)$.

Next, we recall a good pair coming from moduli spaces. Let $f\colon\calC\to\calM_{1,1}$ be the universal curve over the moduli space of elliptic curves. After choosing an $\FF$-point on $X$, the elliptic curve $X$ gets identified as a fibre $\calC_x$ of $f$ at an $\FF$-point $x\in\calM_{1,1}$ classifying the elliptic curve $X$. Set $2\calC_x\colonequals V(I_{\calC_x}^2) \subset \calC$. As $\calM_{1,1}$ is a smooth (Deligne-Mumford) curve, the closed immersion $(X = \calC_x \subset 2 \calC_x)$ is a square-zero thickening whose ideal may be identified with $t_x^\vee \otimes_\FF \calO_X$, where $t_x$ is the tangent space to $\calM_{1,1}$ at $x$ (whence $t_x = H^1(X,T_X)$ by deformation theory). Critically, Igusa's theorem on the reducedness of the supersingular locus~\cite{Igusa} implies that $(X \subset 2\calC_x)$ is good. 

We now prove that $(X \subset \PP^n)$ is good. Let $\calH$ be a suitable Hilbert scheme of elliptic curves in $\PP^n$, and let $g\colon\calX\to\calH$ be the universal elliptic curve, so we have a tautological closed immersion $i_\calH\colon\calX\to\PP^n_\calH$ as well as a distinguished point $y\in\calH(\FF)$ corresponding to $X$ such that the fibre of $i_\calH$ over $y$ identifies with $X \subset \PP^n$. Set $2\calX_y\colonequals V(I_{\calX_y}^2) \subset \calX$. Forgetting the embedding gives a map $\pi\colon\calH\to\calM_{1,1}$ with $\pi(y)=x$; we shall prove this map is smooth at $y$. Granting the smoothness, let us first complete the proof of the theorem. We have a fibre square
\[
\CD
\calX @>>> \calC\\
@VVgV @VVfV\\
\calH @>\pi>> \calM_{1,1},
\endCD
\]
with $\pi$ being smooth at $y\in\calH$, and the vertical maps being relative smooth curves. As tangent vectors can be lifted along smooth maps, it follows that the map 
\[
(X=\calX_y \subset 2\calX_y) \to (X=\calC_x \subset 2\calC_x)
\]
of square-zero thickenings admits a section, so the goodness of one is equivalent to the goodness of the other by descent of goodness, whence $(X=\calX_y\subset 2\calX_y)$ is good by the last paragraph. But we have maps
\[
(X=\calX_y \subset 2\calX_y) \to (X\subset\calX) \to (X\subset\PP^n)
\]
of closed immersions, so the descent of goodness implies that $(X \subset\PP^n)$ is good. 

It remains to prove the smoothness of the map $\pi\colon\calH \to \calM_{1,1}$ from the Hilbert scheme to the moduli space at $y$. As the target is smooth, it suffices to prove the source is smooth and that this map is surjective on tangent spaces. If we write $I_X \subset \calO_{\PP^n}$ for the ideal sheaf of $X$, then the obstruction to smoothness of $\calH$ at $y$ is given by
\[
\Ext^1_X(I_X/I_X^2,\calO_X) = H^1(X, (I_X/I_X^2)^\vee),
\]
while the tangent map $t_{\pi,y}$ identifies with the map 
\[
\Hom_X(I_X/I_X^2, \calO_X) \to H^1(X, T_X)
\]
arising from the standard exact sequence
\[
\CD
0 @>>> T_X @>>> T_{\PP^n}|_X @>>> (I_X/I_X^2)^\vee @>>> 0
\endCD
\]
as the boundary map on global sections. Thus, it is enough to check that the second and third terms in the sequence above have no $H^1$. Now $T_{\PP^n}$ is a quotient of $\calO_{\PP^n}(1)^{n+1}$ by the Euler sequence, so the same is true on restriction to $X$. As $H^{>1}(X,-) = 0$, the functor~$H^1(X,-)$ is right exact, so the vanishing of $H^1(X,\calO_{\PP^n}(1)|_X)$ by Riemann-Roch implies both the desired vanishings.
\end{proof}

%%%%%%%%%%%%%%%%%%%%%%%%%%%%%%%%%%%%%%%%%%%%%%%%%%%%%%%%%%%%%%%%%%%%%%%%

\end{document}